\documentclass[]{elsarticle}

\usepackage[normalem]{ulem}

\usepackage{amssymb,amsmath,amsthm}

\usepackage[utf8]{inputenc}
\usepackage{indentfirst}

\usepackage{todonotes}

\biboptions{numbers,sort&compress}

\usepackage[left=2.5cm, right=2.5cm, top=2cm, bottom=2cm]{geometry}

\usepackage{url}

\usepackage[T1]{fontenc}
\usepackage{lmodern}
\usepackage{microtype}

\usepackage{xcolor,cancel}

\newtheorem{Th}{Theorem}[section]

\newtheorem{Lem}[Th]{Lemma}

\newtheorem{Rem}[Th]{Remark}

\newcommand{\R}{\mathbb{R}}

\newcommand{\cC}{{\mathcal C}}

\newcommand{\cJ}{{\mathcal J}}

\newcommand{\Ga}{\Gamma}

\newcommand{\weakto}{\rightharpoonup}

\numberwithin{equation}{section}

\newcommand{\loc}{\mathrm{loc}}

\begin{document}


\title{Note on semiclassical states for the Schr\"odinger equation with nonautonomous nonlinearities}



\author{Bartosz Bieganowski}
\ead{bartoszb@mat.umk.pl}
\address{Nicolaus Copernicus University, Faculty of Mathematics and Computer Science, ul. Chopina 12/18, 87-100 Toru\'n, Poland}

\author{Jaros\l aw Mederski}
\ead{jmederski@impan.pl}
\address{Institute of Mathematics, Polish Academy of Sciences, ul. \'Sniadeckich 8, 00-656 Warsaw, Poland,\\
CRC 1173 Wavephenomena: Analysis and Numerics,
Departement of Mathematics,
Karlsruhe Institute of Technology (KIT), 
D-76128 Karlsruhe, Germany}

\begin{abstract} 

We consider the following Schr\"odinger equation
$$
- \hslash ^2 \Delta u + V(x)u = \Gamma(x) f(u) \quad \mathrm{in} \ \R^N,
$$
where $u \in H^1 (\R^N)$, $u > 0$, $\hslash > 0$ and $f$ is superlinear and subcritical nonlinear term. We show that if $V$ attains local minimum and $\Gamma$ attains global maximum at the same point or  $V$ attains global minimum and $\Gamma$ attains local maximum at the same point, then there exists a positive solution for sufficiently small $\hslash>0$.

   
\end{abstract}

\begin{keyword}
semiclassical limit \sep variational methods \sep bounded potential \sep nonautonomous nonlinearity
\MSC[2010] 35Q55 \sep 35A15 \sep 35J20 
\end{keyword}

\maketitle

\section{Introduction}
\setcounter{section}{1}

We consider the following semilinear elliptic problem
\begin{equation}\label{eq:1.1}
- \hslash ^2 \Delta u + V(x)u = \Gamma(x) f(u) \quad \mathrm{in} \ \R^N, \ N \geq 1,
\end{equation}
where $u \in H^1 (\R^N)$ and $u > 0$. Equation \eqref{eq:1.1} describes the so-called \textit{standing waves} of the nonlinear, time-dependent Schr\"odinger equation of the form
$$
\mathrm{i} \hslash \frac{\partial \Psi}{\partial t} = - \frac{\hslash^2}{2m} \Delta \Psi + V(x) \Psi - h(x, \Psi).
$$
Solutions of \eqref{eq:1.1} for sufficiently small $\hslash > 0$ are called \textit{semiclassical states}. Recently many papers have been devoted to study semiclassical states, see eg. \cite{delPinoCV, ByeonJeanjean, ByeonTanaka, delPinoJFA, dAvenia, Rabinowitz, Wang} and references therein and most of the concentrate on the case $\Gamma=1$. Our aim is to show that the method introduced by del Pino and Felmer in \cite{delPinoCV} can be also applied to a general class of problems with nonautonomous nonlinearities. 
Moreover, contrary to \cite{delPinoCV}, we do not need the H\"older continuity of the potential $V$, since we do not use the regularity of solutions to show e.g. Lemma \ref{lem:1.5} below. 

We impose the following condition on the potential $V$.
\begin{enumerate}
\item[(V)] $V \in L^\infty (\R^N)$ is continuous and there is a constant $\alpha > 0$ such that $ V(x) \geq \alpha$ for all $x \in \R^N$.
\end{enumerate}

We assume that $f : \R_+ \rightarrow \R$ is of $\cC^1$-class and satisfies the following conditions.

\begin{enumerate}
\item[(F1)] $f(u)=o(u)$ as $u \to 0^+$.
\item[(F2)] $\lim_{u\to\infty} \frac{f(u)}{u^{p-1}} = 0$ for some $2 < p < 2^*$, where $2^* = \frac{2N}{N-2}$ for $N \geq 3$ and $2^* = +\infty$ otherwise.
\item[(F3)] There is $2 < \theta \leq p$ such that $0 < \theta F(u)\leq f(u)u$ for $u > 0$, where $F(u) := \int_0^u f(s) \, ds$.
\item[(F4)] The function $u \mapsto \frac{f(u)}{u}$ is nondecreasing.
\end{enumerate}

Now $\Gamma$ satisfies the  following condition.
\begin{enumerate}
\item[($\Ga$)] $\Ga \in L^\infty (\R^N)$ is continuous and there is a constant $\beta > 0$ such that $\Ga(x) \geq \beta > 0$ for all $x \in \R^N$.
\end{enumerate}

We introduce the following relation between  $V$ and $\Gamma$.

\begin{enumerate}
\item[($\Lambda$)] Assume that there is a bounded, nonempty domain (i.e. open and connected set) $\Lambda \subset \R^N$ such that
\begin{align*}
(\Lambda_1) \ \Ga \mbox{ is } \mathbb{Z}^N\mbox{-periodic and there is } x_{\min} \in \Lambda \mbox{ such that } V(x_{\min}) = \inf_\Lambda V < \min_{\partial \Lambda} V \ \mathrm{and} \ \Ga(x_{\min}) = \sup_{\R^N} \Gamma
\end{align*}
or
\begin{align*}
(\Lambda_2) \ V \mbox{ is } \mathbb{Z}^N\mbox{-periodic and there is } x_{\max} \in \Lambda \mbox{ such that } \Ga(x_{\max}) = \sup_\Lambda \Ga > \max_{\partial \Lambda} \Ga \ \mathrm{and} \ V(x_{\max}) = \inf_{\R^N} V.
\end{align*}
\end{enumerate}
\begin{Rem}
	Without loss of generality we may assume that $|\Ga|_\infty = 1$, where $|\cdot|_k$ denotes the usual $L^k$-norm with $k\geq 1$ or $k=\infty$. Indeed, see that
	$$
	\Gamma(x) f(u) = \frac{\Gamma(x)}{|\Ga|_\infty} |\Ga|_\infty f(u).
	$$
	Taking $\hat{\Gamma} (x) := \frac{\Gamma(x)}{|\Ga|_\infty}$ and $\hat{f} (u) :=  |\Ga|_\infty f(u)$ we see that all conditions are still satisfied and $|\hat{\Ga}|_\infty = 1$. Hence in the rest of the paper we take $|\Ga|_\infty = 1$.
\end{Rem}

Our main result reads as follows.

\begin{Th}\label{th:main1}
Suppose that (V), (F1)--(F4), ($\Ga$) and ($\Lambda$) are satisfied. Then there is $\hslash_0 > 0$ such that for any $\hslash \in (0, \hslash_0)$ the problem \eqref{eq:1.1} has a positive solution $u \in H^1 (\R^N) \cap \cC(\R^N)$ and there are constants $C, \alpha > 0$ such that $u(x) \leq C \exp(-\alpha |x|)$.
\end{Th}

Define the energy functional $\cJ_\hslash : H \rightarrow \R$
$$
\cJ_\hslash (u) := \frac{1}{2}\int_{\R^N} \hslash^2 | \nabla u|^2 + V(x) u^2 \, dx - \int_{\R^N} \Ga(x) F(u) \, dx,
$$
where $H = H^1 (\R^N)$. Note that for any fixed $\hslash > 0$ the quadratic form
$$
H \ni u \mapsto Q_\hslash (u) :=\int_{\R^N} \hslash^2 | \nabla u|^2 + V(x) u^2 \, dx \in \R
$$
is positive definite. Hence for any $\hslash > 0$, $u \mapsto \sqrt{Q_\hslash (u)}$ is well-defined norm on $H$ and equivalent to the classic one in $H$. Let $\theta$ be given by (F3) and fix $k > \frac{\theta}{\theta-2} > 1$. In view of (F3) and (F4) there is $a > 0$ such that $\frac{f(a)}{a} = \frac{\alpha}{k}$. We define
$$
\tilde{f} (u) := \left\{ \begin{array}{ll}
f(u) & \ \mathrm{for} \ u \leq a, \\
\frac{\alpha}{k} u & \ \mathrm{for} \ u > a,
\end{array} \right.
$$
where $\alpha$ has beed introduced in (V) and
$$
g(x, u) := \chi_\Lambda (x) \Gamma(x) f(u) + (1-\chi_\Lambda(x))\Ga(x) \tilde{f}(u).
$$
In what follows we will consider $f$ as a function $f : \R \rightarrow \R$ satisfying (F1)-(F4) for on $\R_+$ and defined as $0$ for $u \leq 0$. Then $g : \R^N \times \R \rightarrow \R$ is a Carath\'eodory function. Moreover the following conditions hold.
\begin{enumerate}
\item[(G1)] $g(x,u) = o(u)$ for $|u| \to 0^+$ uniformly in $x \in \R^N$.
\item[(G2)] $\lim_{u\to\infty} \frac{g(x,u)}{u^{p-1}} = 0$ for some $2<p<2^*$ uniformly in $x \in \R^N$.
\item[(G3)] There is $2 < \theta \leq p$ such that
$$
0 < \theta G(x,u) \leq g(x,u)u \quad \mathrm{for} \ x \in \Lambda, \ u > 0
$$
and
$$
0 \leq 2 G(x,u) \leq g(x,u)u \leq \frac{1}{k} V(x) u^2 \quad \mathrm{for} \ x \not\in \Lambda, \ u > 0.
$$
\item[(G4)] The function $u \mapsto \frac{g(x,u)}{u}$ is nondecreasing on $(0,\infty)$ for all $x \in \R^N$. Moreover, if $x \in \R^N \setminus \Lambda$, the function $u \mapsto \frac{g(x,u)}{u}$ is constant on $[a,\infty)$.
\end{enumerate}

Indeed, (G1), (G2) and (G4) are obvious and we need to check (G3). For $x \in \Lambda$ we have $g(x,u) = \Ga (x) f(u)$ and $G(x,u) = \Ga (x) F(u)$, so the statement follows from (F3). Fix $x \not\in \Lambda$. From (G4) we have
$$
G(x,u) = \int_0^u \frac{g(x,s)}{s} s \, ds \leq \frac{g(x,u)}{u} \int_0^u s \, ds =  \frac{1}{2} g(x,u)u.
$$
Hence $0 \leq 2 G(x,u) \leq g(x,u)u$. Moreover
$$
g(x,u)u = \Ga (x) \tilde{f}(u) u = \left\{ \begin{array}{ll}
\Ga(x) \frac{f(u)}{u} u^2 \leq \Ga(x) \frac{\alpha}{k} u^2, & \quad \mbox{for } u  \leq a, \\
\Ga(x) \frac{\alpha}{k}u^2, & \quad \mbox{for } u  > a.
\end{array} \right.
$$
Thus
$$
g(x,u)u \leq \Ga(x) \frac{\alpha}{k}u^2 \leq \frac{1}{k} \Ga (x) V(x) u^2 \leq \frac{1}{k}  V(x) u^2
$$
and the proof of (G3) is completed. Let $\Phi_\hslash : H \rightarrow \R$ be given by
$$
\Phi_\hslash (u) := \frac{\hslash^2}{2} \int_{\R^N} | \nabla u|^2 \, dx + \frac{1}{2} \int_{\R^N} V(x) u^2 \, dx - \int_{\R^N} G(x,u) \, dx.
$$
From \cite[Lemma 2.1]{delPinoCV} we obtain the following.

\begin{Lem}
The functional $\Phi_\hslash$ possesses a positive critical point $u_\hslash \in H$ such that $\Phi_\hslash (u_\hslash) = c_\hslash$, where
$$
c_\hslash := \inf_{u \in H \setminus \{0\}} \sup_{t \geq 0} \Phi_\hslash(tu).
$$
\end{Lem}

Applying \cite[Theorem 4.2]{Mederski} we see that $u_\hslash$ is continuous and exponentially decays at infinity. Define
\begin{equation}\label{eq:mh}
m_\hslash := \max_{\partial \Lambda} u_\hslash.
\end{equation}

\section{Case ($\Lambda_1$)}

Put $V_0 := \min_\Lambda V$. Let $w \in H$ be a least energy solution to $-\Delta w + V_0 w = \Ga(x) f(w)$, in particular
$$
\underline{c} := I_0 (w) = \inf_{v \in H \setminus \{0\}} \sup_{t \geq 0} I_0 (tv),
$$
where 
$$
I_0 (v) = \frac{1}{2} \int_{\R^N} |\nabla v|^2 + V_0 v^2 \, dx - \int_{\R^N} \Ga(x) F(v) \, dx.
$$
Under $(\Lambda_1)$ we have that $\Ga$ is $\mathbb{Z}^N$-periodic, hence the solution exists (see e.g. \cite{Rabinowitz,Mederski}).
\begin{Lem}
There holds $\Phi_\hslash (u_\hslash) \leq \hslash^N (\underline{c} + o(1))$ as $\hslash \to 0^+$.
\end{Lem}

\begin{proof} 
Let $x_0 \in \Lambda$ be such that $V(x_0) = V_0$. Let $u(x) := w \left( \frac{x-x_0}{\hslash} \right)$. Then $\Phi_\hslash (u_\hslash) \leq \sup_{t > 0} \Phi_\hslash (tu) = \Phi_\hslash (t_0 u)$ for some $t_0 > 0$. See that
\begin{align*}
\Phi_\hslash (t_0 u) &= \frac{t_0^2 \hslash^2}{2} \int_{\R^N} |\nabla u|^2 \, dx + \frac{t_0^2}{2} \int_{\R^N} V(x) u^2 \, dx - \int_{\R^N} G(x, t_0 u) \, dx \\
&= \hslash^N \left( I_0 (t_0 w) + \frac{t_0^2}{2} \int_{\R^N} [ V(x_0 + \hslash x) - V_0 ] w^2 \, dx + \int_{\R^N} \Ga(x)F(t_0 w) - G(x_0 + \hslash x, t_0 w) \, dx \right).
\end{align*}
From the Lebesgue's dominated convergence theorem we have $\int_{\R^N} [ V(x_0 + \hslash x) - V_0 ] w^2 \, dx \to 0$. Note that $G(x, t_0 w) \leq  F(t_0 w)$. Again, from the Lebesgue's dominated convergence theorem the following convergence hold
$$
\int_{\R^N}  G(x_0 + \hslash x, t_0 w) \, dx \to \int_{\R^N} \Ga (x_0) F(t_0 w) \, dx.
$$
Hence $\int_{\R^N} \Ga(x)F( t_0 w) - G(x_0 + \hslash x, t_0 w) \, dx \to \int_{\R^N} \Ga(x)F( t_0 w) - \Ga (x_0)  F(t_0 w) \, dx   \leq 0$ and finally
$$
\Phi_\hslash (t_0 u) \leq \hslash^N \left( I_0 (t_0 w) + o(1) \right) \leq \hslash^N \left( \underline{c} + o(1) \right) .
$$
\end{proof}

\begin{Lem}\label{Lem:1.4}
There is $C > 0$ such that
$$
\int_{\R^N} \hslash^2 |\nabla u_\hslash|^2 + V(x) |u_\hslash|^2 \, dx \leq C \hslash^N.
$$
\end{Lem}

\begin{proof}
Indeed, we have $\Phi_\hslash ' (u_\hslash)(u_\hslash) = 0$, i.e. $\int_{\R^N} \hslash^2 | \nabla u_\hslash|^2 + V(x) |u_\hslash|^2 \, dx = \int_{\R^N} g(x, u_\hslash) u_\hslash \, dx$. On the other hand
\begin{align*}
\frac{1}{2} \int_{\R^N} \hslash^2 | \nabla u_\hslash|^2 + V(x) |u_\hslash|^2 \, dx &=  \Phi_\hslash (u_\hslash) + \int_{\R^N} G(x, u_\hslash) \, dx \\ 
&\leq \hslash^N ( \underline{c} + o(1) ) + \frac{1}{\theta} \int_\Lambda g(x,u_\hslash) u_\hslash \, dx + \frac{1}{2k} \int_{\R^N} V(x) |u_\hslash|^2 \, dx \\
&\leq C_1 \hslash^N + \left( \frac{1}{\theta} + \frac{1}{2k} \right) \int_{\R^N} \hslash^2 | \nabla u_\hslash|^2 + V(x) |u_\hslash|^2 \, dx.
\end{align*}
Hence
$$
\left( \frac{1}{2} - \frac{1}{\theta} - \frac{1}{2k} \right)  \int_{\R^N} \hslash^2 | \nabla u_\hslash|^2 + V(x) |u_\hslash|^2 \, dx \leq C_1 \hslash^N.
$$
Moreover $ \frac{1}{2} - \frac{1}{\theta} - \frac{1}{2k} = \frac{1}{2} \left( \frac{\theta - 2}{\theta} - \frac{1}{k} \right) > 0$ and the proof is finished.
\end{proof}

\begin{Lem}\label{lem:1.5}
If $\hslash_n \to 0^+$  and $(x_n) \subset  \overline{\Lambda}$ are such that $u_{\hslash_n} (x_n) \geq b > 0$, then $\lim_{n\to\infty} V(x_n) = V_0$.
\end{Lem}


\begin{proof}
Assume by contradiction, passing to a subsequence, that $x_n \to \overline{x} \in \overline{\Lambda}$ and $V(\overline{x}) > V_0$. Put $v_n (x) := u_{\hslash_n} (x_n + \hslash_n x)$. Obviously, $v_n \in H$ satisfies the equation
$$
-\Delta v_n  + V(x_n + \hslash_n x) v_n = g(x_n + \hslash_n x, v_n)  \quad \mbox{in} \ \R^N.
$$
From Lemma \ref{Lem:1.4} it follows easily that $(v_n)$ is bounded in $H$ and therefore $v_n \weakto v$ in $H$ for some $v \in H$. Take any $\varphi \in \cC_0^\infty (\R^N)$ and see that
$$
\int_{\R^N} \nabla v_n \cdot \nabla \varphi \, dx \to \int_{\R^N} \nabla v \cdot \nabla \varphi \, dx.
$$
Moreover
$$
\int_{\R^N} V(x_n + \hslash_n x) v_n \varphi \, dx \to \int_{\R^N} V(\overline{x} ) v \varphi \, dx.
$$
Functions $\chi_n (x) := \chi_\Lambda (x_n + \hslash_n x)$ are bounded in $L^t_{\loc} (\R^N)$ for any $1 < t < \infty$, and therefore $\chi_n \weakto \chi$ in $L^t_{\loc} (\R^N)$, where $0 \leq \chi \leq 1$. Hence $v \in H$ is a weak solution to
$$
-\Delta v + V(\overline{x}) v = \overline{g}(x,v) \quad \mathrm{in} \ \R^N,
$$
where
$$
\overline{g}(x,u) = \chi(x) \Ga (\overline{x}) f(u) + (1-\chi(x)) \Ga (\overline{x}) \tilde{f}(u).
$$
The associated energy functional is given by
$$
\overline{\cJ} (u) := \frac{1}{2} \int_{\R^N} |\nabla u|^2 + V(\overline{x}) |u|^2 \, dx - \int_{\R^N} \overline{G}(x,u) \, dx \quad \mathrm{for} \ u \in H^1 (\R^N)
$$
and $\overline{G}(x,u) := \int_0^u \overline{g}(x,s) \, ds$. Since $v$ is a weak solution, we have $\overline{\cJ}' (v) = 0$. Set
$$
\cJ_n (u) := \frac{1}{2} \int_{\R^N} |\nabla u|^2 + V(x_n + \hslash_n x) |u|^2 \, dx - \int_{\R^N} G(x_n + \hslash_n x, u) \, dx \quad \mathrm{for} \ u \in H^1 (\R^N)
$$
and obviously $\cJ_n ' (v_n) = 0$.
See that
$$
\cJ_n (v_n) = \cJ_n (v_n) - \frac12 \cJ_n ' (v_n)(v_n) = \frac12 \int_{\R^N} g(x_n + \hslash_n x , v_n) v_n - 2 G(x_n + \hslash_n x, v_n ) \, dx.
$$
Define
$$
h_n := g(x_n + \hslash_n \cdot , v_n) v_n - 2 G(x_n + \hslash_n \cdot, v_n ).
$$
In view of (G3) we have $h_n \geq 0$. Fix $R > 0$. In view of compact embedding we have $v_n \to v$ in $L^t (B(0,R))$ for $t \in [2, 2^*)$. Moreover $\chi_n \weakto \chi$ in any $L^t (B(0,R))$.
Then
$$
\int_{B(0,R)} h_n \, dx \to \int_{B(0,R)} \overline{g} (x,v)v - 2 \overline{G}(x,v) \, dx
$$
Hence, for every $\delta > 0$ there is $R>0$ large enough such that
\begin{align*}
\frac{1}{2} \int_{\R^N} h_n \, dx &\geq  \frac{1}{2}  \int_{B(0,R)} h_n \, dx \to \frac{1}{2} \int_{B(0,R)} \overline{g} (x,v)v - 2 \overline{G}(x,v) \, dx \\ 
&\geq \frac{1}{2} \int_{\R^N} \overline{g} (x,v)v - 2 \overline{G}(x,v) \, dx - \delta = \overline{\cJ} (v) - \frac12 \overline{\cJ}'(v)(v) - \delta =  \overline{\cJ} (v) - \delta.
\end{align*}
Therefore $\liminf_{n\to\infty} \cJ_n (v_n) \geq \overline{\cJ} (v)$. On the other hand
$
\cJ_n (v_n) = \hslash_n^{-N} \Phi_{\hslash_n} (u_{\hslash_n}) \leq \underline{c} + o(1).
$
Hence $\overline{\cJ} (v) \leq \underline{c}$. Taking into account that $f(u) \geq \tilde{f}(u)$ we have
$$
\overline{\cJ} (v) = \max_{\tau \geq 0} \overline{\cJ} (\tau v) \geq \inf_{u \in H^1 (\R^N), \ u \neq 0} \sup_{\tau > 0} I(\tau u) =: \overline{c},
$$
where
$$
I(u) := \frac12 \int_{\R^N} |\nabla u|^2 + V(\overline{z}) u^2 \, dx - \int_{\R^N} \Gamma(\overline{x}) F(u) \, dx.
$$
On the other hand, taking into account that $V(\overline{x}) > V_0$, there is $\overline{c} > \underline{c}$ -- a contradiction.
\end{proof}

\begin{Lem}\label{Lem:1.6}
There holds $\lim_{\hslash \to 0^+} m_\hslash = 0$, where $m_\hslash$ is given by \eqref{eq:mh}.
\end{Lem}

\begin{proof}
Assume by contradiction that $m_\hslash \not\to 0$. Let $x_\hslash \in \partial \Lambda \subset \overline{\Lambda}$ be such that $u_\hslash(x_\hslash) = m_\hslash$. Then, up to a subsequence we have $u_{\hslash_n} (x_{\hslash_n}) \geq b > 0$ and $x_{\hslash_n} \to x_0 \in \partial \Lambda$. Hence, in view of Lemma \ref{lem:1.5} we gets
$$
\min_{\partial \Lambda} V \leq V(x_0) = V_0 = \inf_\Lambda V < \min_{\partial \Lambda} V,
$$
which is a contradiction.
\end{proof}

\section{Case ($\Lambda_2$)}

Put $\Ga_0 := \max_\Lambda \Ga$. Let $w \in H$ be a least energy solution to $-\Delta w + V(x) w = \Ga_0 f(w)$, in particular
$$
\underline{c} := I_0 (w) = \inf_{v \in H \setminus \{0\}} \sup_{t \geq 0} I_0 (tv),
$$
where
$$
I_0 (v) = \frac{1}{2} \int_{\R^N} | \nabla v|^2 + V(x) v^2 \, dx - \int_{\R^N} \Ga_0 F(v).
$$
Under $(\Lambda_2)$ we have that $V$ is $\mathbb{Z}^N$-periodic, hence the solution exists (see e.g. \cite{Rabinowitz,Mederski}).
\begin{Lem}
There holds $\Phi_\hslash (u_\hslash) \leq \hslash^N (\underline{c} + o(1))$ as $\hslash \to 0^+$.
\end{Lem}

\begin{proof} 
Let $x_{\max} \in \Lambda$ be such that $\Ga(x_{\max}) = \Ga_0$. Let $u(x) := w \left( \frac{x-x_{\max}}{\hslash} \right)$. Then $\Phi_\hslash (u_\hslash) \leq \sup_{t > 0} \Phi_\hslash (tu) = \Phi_\hslash (t_0 u)$ for some $t_0 > 0$. See that
\begin{align*}
\Phi_\hslash (t_0 u) &= \frac{t_0^2 \hslash^2}{2} \int_{\R^N} |\nabla u|^2 \, dx + \frac{t_0^2}{2} \int_{\R^N} V(x) u^2 \, dx - \int_{\R^N} G(x, t_0 u) \, dx \\
&= \hslash^N \left( I_0 (t_0 w) + \frac{t_0^2}{2} \int_{\R^N} [ V(x_{\max} + \hslash x) - V(x) ] w^2 \, dx + \int_{\R^N} \Ga_0 F(t_0 w) - G(x_{\max} + \hslash x, t_0 w) \, dx \right).
\end{align*}
From the Lebesgue's dominated convergence theorem we have
$$
\int_{\R^N} [ V(x_{\max} + \hslash x) - V(x) ] w^2 \, dx \to \int_{\R^N} [ V(x_{\max} ) - V(x) ] w^2 \, dx \leq 0.
$$
Note that $G(x, t_0 w) \leq  F(t_0 w)$. Again, from the Lebesgue's dominated convergence theorem the following convergence hold
$$
\int_{\R^N} \Ga_0 F(t_0 w) - G(x_{\max} + \hslash x, t_0 w) \, dx \to \int_{\R^N} \Ga_0 F(t_0 w) - \Ga(x_{\max})F( t_0 w) \, dx = 0.
$$
Finally $\Phi_\hslash (t_0 u) \leq \hslash^N \left( I_0 (t_0 w) + o(1) \right) \leq \hslash^N \left( \underline{c} + o(1) \right)$.
\end{proof}
Now we can repeat the proof of Lemma \ref{Lem:1.4}, \ref{lem:1.5} and \ref{Lem:1.6}.

\section{Conclusion}


\begin{proof}[Proof of Theorem \ref{th:main1}]
Let $u_\hslash$ be a positive critical point for $\Phi_\hslash$. In view of Lemma \ref{Lem:1.6} there is $\hslash_0$ such that for any $\hslash \in (0, \hslash_0)$ there holds $m_\hslash < a$. Therefore $u_\hslash (x) < a$ for $x \in \partial \Lambda$. Hence, in view of the maximum principle, we have
$$
u_\hslash(x) \leq a \quad \mathrm{for} \ x \in \Lambda.
$$
Take $(u_\hslash - a)_+ := \max\{ u_\hslash - a, 0 \}$ as a test function for $\Phi_\hslash$, i.e. we have $\Phi_\hslash ' (u_\hslash) \left( (u_\hslash - a)_+ \right) = 0$. Thus we get
\begin{equation}\label{eq:c}
\int_{\R^N \setminus \Lambda} \hslash^2 | \nabla (u_\hslash - a)_+ |^2 + c(x) (u_\hslash - a)_+^2 + c(x) a (u_\hslash - a)_+ \, dx = 0,
\end{equation}
where $c(x) = V(x) - \frac{g(x,u_\hslash (x))}{u_\hslash (x)}$. Moreover, for $x \in \R^N \setminus \Lambda$, 
taking into account that $|\Ga|_\infty = 1$, we obtain $\frac{g(x,u_\hslash (x))}{u_\hslash (x)} \leq \frac{\alpha}{k}$. Therefore $c(x) > 0$ for $x \in \R^N \setminus \Lambda$.
Hence all summands in \eqref{eq:c} are zero. In particular $(u_\hslash - a)_+ = 0$ and $u_\hslash(x) \leq a$ for $x \in \R^N \setminus \Lambda$. Hence $g(x,u_\hslash(x)) = \Ga (x) f(u_\hslash(x))$ and $u_\hslash$ is a solution of \eqref{eq:1.1}.
\end{proof}

\textbf{Acknowledgements.} Bartosz Bieganowski was partially supported by the National Science Centre, Poland (Grant No. 2017/25/N/ST1/00531). Jarosław Mederski  was partially supported by the National Science Centre,
Poland (Grant No. 2014/15/D/ST1/03638) and by the Deutsche Forschungs\-gemeinschaft (DFG) through CRC 1173.

\end{document}